\documentclass[11pt]{amsart}
\usepackage{amsmath}
\usepackage{amssymb}
\usepackage{tabularx}
\usepackage{enumerate}
\usepackage{graphicx}
\usepackage{texdraw}
\usepackage{mathrsfs}
\usepackage{epstopdf}

\usepackage{mathrsfs}

\usepackage{amsfonts,amssymb,amsmath}
\usepackage{epsfig}

\usepackage{color}

\def\caL{\mathcal{L}}

\makeatletter
\@addtoreset{equation}{section}

\makeatother
\newtheorem{thm}{Theorem}[section]
\newtheorem{lem}[thm]{Lemma}

\newtheorem{prop}[thm]{Proposition}
\newtheorem{remark}[thm]{Remark}

\newcommand{\R}{\mathbb{R}}
\newcommand{\U}{\mathcal{U}}

\begin{document}
\title[Entire Solutions on the Half Line]{Entire Solutions of the Fisher-KPP Equation
on the Half Line$^\S$}
 \thanks{$\S$ This research was partly supported by the NSFC (No.  11671262) and JSPS KAKENHI Grant Number, 18H01139, 26247013 and JST CREST Grant Number JPMJCR14D3.}
\author[B. Lou, J. Lu and Y. Morita]{Bendong Lou$^{\dag, *}$, Junfan Lu$^\ddag$ and Yoshihisa Morita$^{\natural}$}
\thanks{$\dag$ Mathematics \& Science College, Shanghai Normal University, Shanghai 200234, China.}
\thanks{$\ddag$ School of Mathematical Sciences, Tongji University, Shanghai 200092, China.}
\thanks{$\natural$ Department of Applied Mathematics and Informatics, Ryukoku University, Seta Otsu 520-2194, Japan.}
\thanks{$*$ Corresponding author.}
\thanks{{\bf Emails:} {\sf lou@shnu.edu.cn} (B. Lou),  {\sf 1410541@tongji.edu.cn} (J. Lu), {\sf morita@rins.ryukoku.ac.jp} (Y. Morita)}
\date{}

\begin{abstract}
In this paper we study the entire solutions of the Fisher-KPP equation $u_t=u_{xx}+f(u)$ on the half line
$[0,\infty)$ with Dirichlet boundary condition at $x=0$.  (1). For any $c\geq 2\sqrt{f'(0)}$,
we show the existence of an entire solution $\U^c(x,t)$ which connects the traveling wave solution $\phi^c(x+ct)$ at $t=-\infty$
and the unique positive stationary solution $V(x)$ at $t=+\infty$; (2). We also construct an entire solution $\U(x,t)$
which connects the solution of $\eta_t =f(\eta)$ at $t=-\infty$ and $V(x)$ at $t=+\infty$.
\end{abstract}

\maketitle

\baselineskip 17pt

\section{Introduction}
Consider the following reaction-diffusion equation:
\begin{equation}\label{eq}
u_t = u_{xx} +f(u),
\end{equation}
where $f\in C^1 ([0,1])$ is a Fisher-KPP type of nonlinearity:
\begin{equation}\label{Fisher-KPP}
f(0)=f(1)=0, \quad f'(1)<0< f'(0),\quad f(u) >0 \mbox{ and } f'(u)\leq f'(0) \mbox{ for } u\in (0,1).
\end{equation}
%In this paper we always write
%$$
%\mu:= \sqrt{f'(0)}.
%$$
It is well known that, for each real number $c\geq c_0: = 2\sqrt{f'(0)} $, the equation \eqref{eq} admits a
unique traveling wave solution $u= \phi^c (x+ct)$, where $\phi^c (z)$ satisfies
\begin{equation}\label{phi eq}
\phi_{zz} -c\phi_z +f(\phi)=0,\quad \phi(-\infty)=0,\quad \phi(\infty)=1,\quad \phi_z(z)>0 \mbox{ for } z\in\R.
\end{equation}
Traveling wave solutions are a special kind of {\it entire solutions} (that is, solutions of \eqref{eq}
defined for all $t\in \R$),
and play a key role in understanding the dynamics of the equation.
In order to figure out the complete global dynamics, however, we need to investigate other types of entire solutions.

Since the pioneering works in \cite{HN} and \cite{Y}, there are huge number of works on entire solutions to
reaction-diffusion equations. In particular, in one space dimension, \cite{CG, CGN, FMN, GM, HN, MN1}, etc. studied entire
solutions of the reaction-diffusion equations like \eqref{eq} with monostable or bistable type of nonlinearities.
The readers also see the survey paper \cite{MN2} and the references cited in.
Meanwhile, some authors studied the entire solutions for lattice differential equations (cf. \cite{WLR2, WSY})
and for equations with nonlocal or delayed terms (cf. \cite{LSW, SLW,WLR1}), etc.

In this paper we are interested in the entire solutions of the equation \eqref{eq} on the half line:
\begin{equation}\label{p}
\left\{
\begin{array}{ll}
 u_t = u_{xx} + f(u),  & x>0,  \ t>0,  \\
 u(0,t)= 0 , &  t>0.
\end{array}
\right.
\end{equation}
This problem has no traveling wave solutions but
a unique positive stationary solution $u =V(x)$ with
$$
V(0)=0,\quad V(\infty)=1,\quad V'(x)>0 \mbox{ for } x\geq 0
$$
(see details in the next section) and, by a similar argument as in \cite{AW1, DM} one sees that,
any solution $u(x,t)$ of \eqref{p} starting from a nonnegative initial data $u(x,0)$ converges as $t\to \infty$ to
$V(x)$ in the topology of $L^\infty_{loc} ([0,\infty))$. Hence, any positive entire solution $\U(x,t)$,
if it exists, also satisfies $\U(\cdot, t)\to V(\cdot)$ as $t\to \infty$.
To distinguish entire solutions, we classify the $\alpha$-limit
of $\U(x,t+t_n)$ for any $t_n \to -\infty$.
In fact, we construct two types of entire solutions for \eqref{p}.
Each of the first type connects a traveling wave solution $\phi^c (x+ct)$ at $t=-\infty$ with $V(x)$ at $t=\infty$,
while the second type connects the solution $\eta (t)$ of the ordinary differential equation $\eta'(t)=f(\eta)$
at $t=-\infty$ with $V(x)$ at $t=\infty$.

On the first type of entire solutions we have the following result.

%%%%%%%%%%%%%%%%%%%%%%%%%%

\begin{thm}\label{thm:I}
Assume \eqref{Fisher-KPP}. Then for each $c\geq c_0 := 2 \sqrt{f'(0)} $ and any $\theta \in \R$,
the problem \eqref{p} has an entire solution $\U^c (x,t)$ which satisfies
\begin{equation}\label{property of Uc}
\U^c_x  (x,t)>0 \mbox{ for } x> 0,\ t\in \R; \qquad  \U^c (\infty,t) =1 \mbox{ for } t\in \R,
\end{equation}
\begin{equation}\label{Uc to V}
\sup\limits_{x\in [0,\infty)} | \U^c (x,t) - V(x)| \to 0 \mbox{ as } t\to \infty,
\end{equation}
and
\begin{equation}\label{Uc to phi}
\sup\limits_{x\in [0,\infty)} |\U^c (x,t) -  \phi^c (x+ct-\theta)| \to 0 \mbox{ as  } t\to -\infty.
\end{equation}
\end{thm}

To obtain the second type of entire solution we need a slightly stronger condition on $f$:
\begin{equation}\label{additional cond}
  f\in C^2([0,1]),\quad f''(u)\leq 0 \mbox{ in } [0,1].
\end{equation}

\begin{thm}\label{thm:II}
Assume \eqref{Fisher-KPP} and \eqref{additional cond}. Then the problem \eqref{p} has an entire solution $\U (x,t)$
with the following properties:
\begin{equation}\label{monotone of U}
\U_t(x,t)>0,\quad \U_x(x,t)>0, \quad \U_{xx}(x,t)<0 \mbox{ for all } x> 0,\; t\in \R;
\end{equation}
\begin{equation}\label{infinity = 1}
\U (\infty,t) =1 \mbox{ for all } t\in \R;
\end{equation}
\begin{equation}\label{U to V}
\sup_{x\in [0,\infty)} | \U (x ,t) - V(x)|  \to 0 \mbox{ as } t\to \infty;
\end{equation}
and, with $\xi_m (s)$ denoting the $m$-level set of $\U(\cdot,s)$ for each given $m\in (0,1)$,
\begin{equation}\label{t to -infinity}
\left\{
 \begin{array}{l}
 \xi_m(s)\to \infty \mbox{ and } -\xi'_m (s) \to \infty \mbox{ as } s\to -\infty,\\
 \U(x+\xi_m(s),t+s)\to \eta_m (t) \mbox{ as } s\to -\infty, \quad  \mbox{ in } C^{2,1}_{loc} (\R \times \R) \mbox{ topology},
 \end{array}
 \right.
\end{equation}
where $\eta_m(t)$ is the unique solution of the initial value problem:
\begin{equation}\label{ivp}
\eta' (t) =f(\eta),\quad \eta(0)=m.
\end{equation}
\end{thm}
\noindent

These two theorems show that, in a phase space (see Figure 1), the orbit of a first type of entire solution $\U^c (\cdot,t)$ connects
that of $\phi^c (\cdot +c t+\theta)$ at $t=-\infty$ and that of $V(\cdot)$ at $t=\infty$; % (see $\Gamma^{c_i}$ in Figure 1);
while the orbit of the second type of entire solution $\U $ connects that of $\eta_m(t)$ at $t=-\infty$ and
that of $V(\cdot)$ at $t=\infty$. % (see $\Gamma^*$ in Figure 1).
Since, in the phase space, the points $\phi^c (\cdot)$ form a \lq\lq one-dimensional continuous manifold" as $c$ increasing
from $c_0 $ to $\infty$ (see $\mathcal{M}$ in Figure 1), and since $\phi^c(x+\hat{\xi}_m) \to m =\eta_m(0)$ as $c\to \infty$, locally
uniformly in $x\in \R$,  where $\hat{\xi}_m$ is the $m$-level set  of $\phi^c(\cdot)$, we can roughly say that
$\phi^c(x+ct +\hat{\xi}_m)\to \eta_m(t)$ as $c\to \infty$.

We note that our setting, the Dirichlet boundary condition on the Fisher-KPP equation in the half-line,
imposes a strong restriction on the structure of the solutions.
In fact, the readers could find various types of entire solutions in \cite{HN,MN2}
for reaction-diffusion equations in the whole space while in our problem
it seems that all the possible entire solutions consist of the stationary solutions ($0$ and $V$),
the two types of entire solutions ($\U^c$ and $\U$) and their temporal translations.
It, however, not so simple to determine all the entire solution even in our simple setting and
a further study would be required for the desired assertion.

\begin{figure}
\begin{center}
\end{center}
\caption{Orbits of the entire solutions on the phase space.}
\end{figure}

Our approach to these theorems are quite different. To prove Theorem \ref{thm:I} we directly construct a pair of
sub- and supersolutions: $\underline{u}(x,t)=\phi^c(x+ct+\theta)-\rho(t)$ and
$\overline{u}(x,t)=\phi^c(x+ct+\theta)$ over $t\leq 0$, for some $\theta\in \R$ and $\rho(t)\searrow 0\
(t\to -\infty)$.
Then the entire solution can be constructed in between as in \cite{FMN, GM} (see details in section 3).
To prove Theorem \ref{thm:II} we consider a sequence of initial-boundary value problems with
initial data $\approx \frac1n$, each solution $u_n(x,t)$ is concave in $x$.
A subsequence of $\{u_n\}$ is proved to converge to a second type of entire solution $\mathcal{U}(x,t)$, whose
concavity is a main feature to distinguish it from the first type of entire solutions.
In particular, to prove the property $\mathcal{U}(\infty, t)\equiv 1$
(despite of the fact $u_n(\infty, 0)=\frac1n$), we have to do some precise estimate for the
solutions of several related linear problems (see details in subsection 4.2).
This is quite different from the approach in section 3.

This paper is arranged as the following. In section 2, as preliminaries, we present the positive
stationary solution $V(x)$ of \eqref{p} and traveling wave solutions $\phi^c(x+ct)$ of \eqref{eq}. In section 3,
we show the existence of the first type of entire solutions $\U^c (x,t)$ for any $c\geq c_0 $, and prove Theorem \ref{thm:I}.
In section 4 we construct the second type of entire solution $\U (x,t)$ and prove Theorem \ref{thm:II}.

%%%%%%%%%%%%%%%%%%%%%%%
%%%%%%%%%%%%%%%%%%%%%%%%
\section{Stationary Solutions and Traveling Wave Solutions}

In this section we present the positive stationary solution $V(x)$ and traveling wave solutions of \eqref{p}$_1$.
Consider the following equation
\begin{equation}\label{phase sol}
q''(z) - c q'(z)+f(q)=0,\quad q(z) \geq   0 \quad \mbox{ for } z\in J,
\end{equation}
where $J$ is some interval in $\R$. Note that a nonnegative stationary solution $u$ of \eqref{p}$_1$ solves
\eqref{phase sol} with $c = 0$ in $J=(0,\infty)$, and a nonnegative traveling wave solution $u(x,t)= q(x+ct)$ of \eqref{eq}
solves \eqref{phase sol} in $\R$.
The equation \eqref{phase sol} is equivalent to the system
\begin{equation}\label{q-p}
\left\{
\begin{array}{l}
q'(z)=p,\\ p'(z)=  c  p- f(q).
\end{array}
\right.
\end{equation}
A solution $(q(z), p(z))$ of this system traces out a trajectory in the $q$-$p$ phase plane.
It is easily seen that $(0,0)$ and $(1,0)$ are two equilibrium
points of the system \eqref{q-p}.
The eigenvalues of the corresponding linearizations at these points are
\begin{equation}\label{eigenvalues}
\lambda_{0}^{\pm} (c)=\frac{c \pm\sqrt{c^2 -4f'(0)}}{2}\ \ (\mbox{at } (0,0))\quad
\mbox{and} \quad
\lambda_{1}^{\pm} (c)=\frac{c \pm\sqrt{c^2 -4 f'(1)}}{2}\ \ (\mbox{at } (1,0)),
\end{equation}
respectively. Since $ f'(0)>0$ and $f'(1)<0$, $(1,0)$ is always a saddle point, however, $(0,0)$ is a center
or a focus when $0\leq c < c_0 =2\sqrt{f'(0)}$, and it is an unstable node when $c \geq c_0 $.
Using the phase plane analysis (cf. \cite{AW1, DL, GLZ}), it is not difficult to give the
solutions of \eqref{phase sol}. We list two types of them, which will be used in this paper.

\medskip
{\rm (I)} \textbf{Positive stationary solution $V(z)$ on the half line}.
When $c =0$, the system \eqref{q-p} can be solved explicitly. In particular, the trajectory
tending to $(1,0)$ in the domain $\{(q,p) \mid 0<q<1,\ p>0\}$ is given by
$p= \sqrt{2 \int_q^1 f(s) ds}$ (see $\Gamma_1$ in Figure 2 (a)), which corresponds to a solution $q=V(z)$
of \eqref{phase sol} with $c=0$. It satisfies (by shifting its zero to $z=0$)
$$
V (0)=0, \quad V (\infty)=1,\quad V' (z)>0 \mbox{ for } z\in [0,+\infty).
$$
%By \eqref{eigenvalues} \blue{there is a positive constant $A_1$ such that} %%we have
%\begin{equation}\label{slope of V}
%1- V(z) = %%O(1)  e^{\lambda_1^- (0) z} = O(1) e^{-\sqrt{-f'(1)} \; z} \mbox{ as } z\to \infty.
%(A_1+o(1))e^{\lambda_1^- (0) z} = (A_1+o(1))e^{-\sqrt{-f'(1)} \; z} \mbox{ as } z\to \infty.
%\end{equation}

\medskip
{\rm (II)} \textbf{Strictly increasing solutions $\phi^c (z)$ in $\R$ in case $c \geq c_0$}.
It is well known that (cf. \cite{AW1,KPP}) for any $c\geq c_0 $, the equation \eqref{phase sol} has a solution
$q=\phi^c (z)$ satisfying
\begin{equation}\label{tw R}
\phi^c (-\infty )=0,\quad \phi^c (\infty)=1,\quad \phi^c (0)=\frac12,\quad  \phi^c_z (z)>0 \mbox{ for } z\in \R
\end{equation}
(see $\Gamma_2$ in Figure 2 (b)).
For each $\phi^c (z)$, it is clear that $u=\phi^c (x+c t)$ is a traveling wave solution of \eqref{eq}.
Moreover, when $c = c_0 = 2 \sqrt{f'(0)} $, we have
$$
\phi^{c_0 } (z) \sim |z| e^{ c_0  z/2} \mbox{ as }z\to -\infty ;
$$
when $c> c_0 $, we have
$$
\phi^{c} (z) \sim e^{\lambda_c  z} \mbox{ as }z\to -\infty,\qquad \mbox{with } \lambda_c := \lambda_{0}^- (c)=\frac12 (c - \sqrt{c^2 -4 f'(0)}) >0.
$$

\begin{figure}[!htbp]
\begin{center}
\end{center}
\caption{\small{Trajectories of the system \eqref{q-p}.
(a) $c=0$;  (b) $c\geq  2\sqrt{f'(0)}$.}}
\end{figure}

%%%%%%%%%%%%%%%%%%%%%%%%%%%%%%%%%%%%%%%%%%%%%%%%%%%%%
%%%%%%%%%%%%%%%%%%%%%%%%%%%%%%%%%%%%%%%%%%%%%%%%%%%%%
%%%%%%%%%%%%%%%%%%%%%%%%%%%%%%%%%%%%%%%%%%%%%%%%%%%%%
%%%%%%%%%%%%%%%%%%%%%%%%%%%%%%%%%%%%%%%%%%%%%%%%%%%%%

%%%%%%%%%%%%%%%%%%%%%%%%%%
\section{The First Type of Entire Solutions}
%%%%%%%%%%%%%%%%%%%%%%%%%%
We prove the existence of $\mathcal{U}^c$ and \eqref{Uc to phi} in Theorem \ref{thm:I}.
We extend $f(u)$ outside the interval $[0,1]$ so that
\begin{equation}
\label{fprime}
f'(u)=f'(0), \qquad u<0
\end{equation}
In fact, since we will be able to obtain the entire solution taking the values in $(0,1)$, this modification does not
affect the desired entire solution.
We define
\[
\caL[u]:=u_t-u_{xx}-f(u).
\]
Then $c\lambda_c=f'(0)b_c $, where $\lambda_c$ is defined as above and
\begin{equation*}
b_c:= \frac{2c}{c+\sqrt{c^2-4f'(0)}} \quad(2\sqrt{f'(0)}\leq c<\infty)
\end{equation*}
is strictly monotone decreasing in $c$ and $f'(0)<c\lambda_c\leq 2f'(0)$.
Indeed,
\begin{eqnarray*}
&&
\frac{d}{dc}\left(\frac{c}{c+\sqrt{c^2-4f'(0)}}\right)=
\frac1{c+\sqrt{c^2-4f'(0)}}-\frac{c(1+c/\sqrt{c^2-4f'(0)})}{(c+\sqrt{c^2-4f'(0)})^{2}} \\
%\nonumber \\
&&
=\frac{1}{(c+\sqrt{c^2-4f'(0)})^{2}}\left(\sqrt{c^2-4f'(0)}-c^2/\sqrt{c^2-4f'(0)}\right) \\
%\nonumber \\
&&
=\frac{-4f'(0)}{(c+\sqrt{c^2-4f'(0)})^{5/2}}<0.
%\nonumber
\end{eqnarray*}
We note that $b_{c_0}=2$.

We first consider the case $c\in(c_0,\infty)$.
Then there is a positive number $A_c$ such that
\[
0<\phi^{c}(z)\leq A_ce^{\lambda_c z},\quad -\infty<z\leq 0.
\]
For arbitrarily given $\theta \in \R$, we put
\begin{eqnarray}
\label{ou}
&&
\overline{u}(x,t):=\phi^c(x+ct-\theta), \\
&&
\underline{u}(x,t):=\phi^c(x+ct-\theta)-\rho(t),
\label{uu}
\end{eqnarray}
where
\begin{equation}
\label{etat}
\rho(t) := A_ce^{-\lambda_c\theta} e^{c\lambda_ct}\qquad(-\infty<t\leq0).
\end{equation}

It is clear that $\overline{u}$ of \eqref{ou} is a supersolution of \eqref{p}
not only for $c>c_0$ but also for $c=c_0$.
Plug $\underline{u}$ of \eqref{uu} into $\caL[u]$ to yield
\begin{eqnarray*}
\caL[\underline{u}]&=&c(\phi^c)'-\dot{\rho}(t)-(\phi^c)''-f(\phi^{c}-\rho(t)) \\
& = &-\dot{\rho}(t)+f(\phi^{c})-f(\phi^{c}-\rho(t)) \\
&=&-\dot{\rho}(t) -\int_0^1 f'(\phi^{c}-s\rho(t))ds(-\rho(t))\\
& = & \rho(t) \Big[ -c\lambda_c+\int_0^1 f'(\phi^{c}-s\rho(t))ds \Big] \\
&\leq&\rho(t)[-c\lambda_c+ f'(0)] <0.
\end{eqnarray*}
Moreover,
\begin{eqnarray*}
\underline{u}(0,t)&=&\phi^{c}(ct-\theta)-\rho(t)\leq A_c e^{\lambda_c(ct-\theta)}-\rho(t) = 0\qquad(t\leq \theta/c),
\end{eqnarray*}
which implies that $\underline{u}$ is a subsolution in $t\in(-\infty, \theta/c]$.

Next in the case $c=c_0$ we have
\[
c_0 \lambda_{c_0}=2f'(0),
\]
and
\[
0<\phi^{c_0}(z)\leq A_{c_0}|z|e^{\lambda_{c_0}z}, \quad -\infty<z\leq0,
\]
for a positive constant $A_{c_0}$.
We set
\[
\rho_* (t) :=\rho^0_*  e^{pt}\quad(t\leq0),\quad  p:=2f'(0)-\delta, \quad 0<\delta<f'(0).
\]
Then, in a similar way as in the previous case
$\underline{u}(x,t):=\phi^{c_0}(x+c_0t-\theta)-\rho_* (t)$ enjoys
\[
\caL[\underline{u}]\leq \rho_*(t)[-p+f'(0)]=\rho_*(t)[-f'(0)+\delta]<0.
\]
In order to show $\underline{u}(0,t)\leq0$, we compute
\begin{eqnarray*}
\underline{u}(0,t) & \leq & A_{c_0}|c_0 t-\theta |e^{\lambda_{c_0}(c_0 t-\theta)}-\rho^0_* e^{pt} \\
&\leq&
e^{pt}[A_{c_0}(|c_0 t|+|\theta|)e^{2f'(0) t-pt}e^{-\lambda_{c_0}\theta}-\rho^0_* ] \\
&=&
e^{pt}A_{c_0}[|c_0 t|e^{\delta t}e^{-\lambda_{c_0}\theta}+|\theta|e^{-\lambda_{c_0}\theta}e^{\delta t}   -\rho^0_* /A_{c_0}]\\
&\leq&
e^{pt}A_{c_0}[\{\sup_{t\leq0}|c_0 t|e^{\delta t}\}e^{-\lambda_{c_0}\theta}+|\theta|e^{-\lambda_{c_0}\theta}  -\rho^0_* /A_{c_0}], \quad (t\leq \theta/c_0).
\end{eqnarray*}
Thus, taking $\rho_*^0$ as
\[
\rho^0_* =A_{c_0}e^{-\lambda_{c_0}\theta}(\sup_{t\leq0}|c_0 t|e^{\delta t}+|\theta|)
\]
yields $\underline{u}(0,t)\leq0~(t\leq\theta/c_0)$.

In the sequel, for any $c\geq c_0$ and arbitrarily given $\theta$ we have obtained the sub-super solution pairs,
by which the exisitence of a solution $\mathcal{U}^c(x,t)$ sandwiched by
$\underline{u}(x,t)$ and $\overline{u}(x,t)$ in $t\in(-\infty,\theta/c]$ is shown
in a similar way as in \cite{FMN, GM}. This solution satisfies the desired asymptotic behavior as $t\to -\infty$
and can be extended to the whole time by the theorem of Cauchy problem.
This concludes the proof.
\qed

\begin{remark}\rm
We can find a similar subsolution to $\underline{u}$ of \eqref{uu} in \cite{FM}, where
they utilize it to prove the asymptotic stability of the traveling front solution to the bistable
reaction-diffusion equation.
On the other hand the present study is related to the asymptotic behavior as $t\to-\infty$.
We, however, see that this type of subsolution is quite useful.
\end{remark}

%%%%%%%%%%%%%%%%%%%%%%%%%%%%%%%%%%%%%%%%%%%%%%%%%%%%%
%%%%%%%%%%%%%%%%%%%%%%%%%%%%%%%%%%%%%%%%%%%%%%%%%%%%%
%%%%%%%%%%%%%%%%%%%%%%%%%%%%%%%%%%%%%%%%%%%%%%%%%%%%%
%%%%%%%%%%%%%%%%%%%%%%%%%%%%%%%%%%%%%%%%%%%%%%%%%%%%%

\section{The Second Type of Entire Solution}
In this section, we always assume \eqref{Fisher-KPP} and \eqref{additional cond}. We first construct a
second type of entire solution in subsection 4.1, and then study its properties in subsection 4.2.
Finally, in subsection 4.3 we study the limit of $\mathcal{U}$ as $t\to -\infty$ and prove
Theorem \ref{thm:II}. For simplicity, in what follows we write
\begin{equation}\label{def-mu}
\mu := \sqrt{f'(0)} = \frac{c_0}{2}.
\end{equation}

\subsection{Construction of the second type of entire solution}
By \eqref{Fisher-KPP}, there exists a large integer $N$ such that
$$
f'(u) >\frac12 f'(0) =\frac12 \mu^2 , \quad u\in \Big[ 0,\frac1N \Big].
$$
For each positive integer $n$, define
$$
\psi_n (x) := \left\{
 \begin{array}{ll}
\displaystyle \frac{1}{(n+N)\pi} \Big[ \sin \frac{\mu x}{\sqrt{2}} + \frac{\mu x}{\sqrt{2}} \Big] , &
                  \displaystyle 0\leq x\leq \frac{\sqrt{2}\pi}{\mu },\\
\displaystyle \frac{1}{n+N}, & \displaystyle x\geq \frac{\sqrt{2}\pi}{\mu}.
 \end{array}
 \right.
$$
Clearly, $\psi_n(x) \in C^2 ([0,\infty))$ and
\begin{equation}\label{property of phin}
\psi'_n (x)\geq 0,\quad   \psi''_n (x)\leq 0,\quad
\psi''_n (x) + f(\psi_n (x)) \geq \psi''_n (x) + \frac{\mu^2}{2} \psi_n (x) \geq 0  \mbox{ \ \ for } x\geq 0.
\end{equation}
We construct the second type of entire solution of \eqref{p} by using the solutions of the following
initial-boundary value problems:
\begin{equation}\label{IBVP1}
\left\{
 \begin{array}{ll}
 u_t = u_{xx} + f(u), & x>0,\ t>0,\\
 u(0,t)=0, & t>0,\\
 u(x,0) = \psi_n (x), & x\geq 0.
\end{array}
\right.
\end{equation}

%%%%%%%%%%%%%%%%%%%%%%%
%%%%%%%%%%%%%%%%%%%%%%%

\begin{lem}\label{lem:IBVP1}
For each positive integer $n$, the solution $u_n(x,t)$ of \eqref{IBVP1} exists for all $t>0$, and it satisfies
\begin{equation}\label{property of un}
(u_n)_t (x,t)>0,\quad  (u_n)_x (x,t)>0,\quad (u_n)_{xx} (x,t)<0 \mbox{\ \ for } x>0,\ t>0;
\end{equation}
\begin{equation}\label{convergence un to V}
u_n (\cdot, t) \to V(\cdot) \mbox{ as } t\to \infty, \quad \mbox{in } C^2_{loc} ([0,\infty)) \mbox{ topology};
\end{equation}
and
\begin{equation}\label{hnt < eta}
u_n (x, t) \leq \eta_{\frac{1}{n+N}} (t) \mbox{ for all } x\geq 0,\ t>0,
\end{equation}
where $\eta_{\frac{1}{n+N} }(t)$ is the solution of the initial value problem \eqref{ivp} with $m=\frac{1}{n+N}$.
\end{lem}

\begin{proof}
By the standard parabolic theory, the classical solution $u_n(x,t)$ of the problem
\eqref{IBVP1} exists globally.
The first two inequalities in \eqref{property of un} follow from \eqref{property of phin} and the strong
maximum principle easily. To show the third inequality, we see that $\zeta(x,t):= (u_n)_{xx}(x,t)$ satisfies
$$
\zeta_t =\zeta_{xx} + f'(u_n)\zeta + f''(u_n) (u_n)^2_x   \leq \zeta_{xx} + f'(u_n)\zeta
$$
by the assumption \eqref{additional cond}. Moreover,
$$
\zeta(0,t) = (u_n)_{xx}(0,t) = (u_n)_t(0,t) - f(u_n(0,t))=0,\quad t>0.
$$
Hence
$$
\zeta(x,t)< 0 \quad \mbox{for } x> 0 \mbox{ and } t>0
$$
by the strong maximum principle and the fact $\psi''_n (x)\leq,\not\equiv 0$
in \eqref{property of phin}.

Since $(u_n)_t (x,t)>0$, by parabolic estimates, $u_n(\cdot,t)$ converges as $t\to \infty$
to a positive stationary solution $\widetilde{V}(x)$. The uniqueness of stationary solutions to our equation implies
$\widetilde{V}(x)\equiv V(x)$ and so
\begin{equation}\label{un to V}
u_n (\cdot, t) \to V(\cdot) \mbox{ as } t\to \infty,
\end{equation}
in the topology of $C^2_{loc} ([0,\infty))$.

%\footnote{\red{By using odd symmetric extension of the equation to the whole line,
%we obtain the Cauchy problem of the Allen-Cahn equation.
%Then \cite[Theorem 1.1]{DM} (which says that any bounded time-global solution of the
%Cauchy problem of $u_t=u_{xx} +f(u)$ for general $f$ converges to a stationary one) is applied to indicate the convergence $u_n\to V$.
%See also \cite{CLZG, DL} for related results.}}.

Finally, \eqref{hnt < eta} follows from the fact that $\eta_{\frac{1}{n+N} }(t)$ is a supersolution of \eqref{IBVP1}.
\end{proof}

Set
\begin{equation}\label{def:tn to infty}
t_n := \min \{t>0\; |\; (u_n)_x  (0,t)  =   V'(0)/2 \}.
\end{equation}
We claim that $t_n \to \infty$ as $n\to \infty$. In fact, from \eqref{hnt < eta} it is easily to know that $u_n(x,t)\to 0$
as $n\to \infty$, locally uniformly in $t\in [0,+\infty)$ and uniformly in $x\geq 0$. This holds true for $(u_n)_x$, thanks
to parabolic estimates. The claim is then proved. Define
$$
U_n (x,t) := u_n (x, t+t_n) \mbox{ for } x\geq 0,\ t> -t_n.
$$
By Lemma \ref{lem:IBVP1} we have
\begin{equation}\label{prop Un}
(U_n)_x (0,0)=\frac12 V'(0),\quad (U_n)_t (x,t) >0,\quad (U_n)_x (x,t) >0,\quad (U_n)_{xx} (x,t) <0 \mbox{\ \ for } x> 0,\ t>-t_n.
\end{equation}
For any $\alpha \in (0,1)$, $T>0$, $X> 0$ and any integer $n$ large such that $t_n >T$, by the $L^p$ estimate we have
$$
\|U_n (\cdot,\cdot)\|_{W^{2,1}_4 ([0,X]\times [-T,T])} \leq C_1,
$$
for some $C_1$ depending on $T$ and $X$ but not on $n$, where, for any bounded domain $Q\subset \{(x,t)\mid x\geq 0, t\in \R\}$,
$W^{2,1}_4 (Q) $ denotes the Sobolev space $\{u\in L^4(Q)\mid \|u\|_{L^4(Q)} + \|u_x\|_{L^4(Q)} + \|u_{xx}\|_{L^4(Q)} + \|u_t\|_{L^4(Q)} <\infty\}$.
Using the embedding theorem we have
$$
\|U_n (\cdot,\cdot)\|_{C^{\alpha, \alpha/2} ([0,X]\times [-T,T])} \leq C_2 \|U_n (\cdot,\cdot)\|_{W^{2,1}_4 ([0,X]\times [-T,T])} \leq C_2 C_1,
$$
for some $C_2$ independent of $n$. By the Schauder estimate we derive
\begin{equation}\label{Schauder est}
\|U_n (\cdot,\cdot)\|_{C^{2+\alpha, 1+\alpha/2}([0,X]\times [-T,T])} \leq C_3,
\end{equation}
for some $C_3$ depending on $\alpha, T, X$ but not on $n$. Therefore, there exists a sequence $\{n_i\}$ of $\{n\}$ and
a function $\U (x,t)\in C^{2+\alpha, 1+\alpha/2} ([0,X] \times [-T,T])$ such that
$$
\|U_{n_i} (\cdot,\cdot) - \U (\cdot,\cdot) \|_{C^{2,1}([0,X]\times [-T,T])} \to 0, \mbox{ as }i\to \infty.
$$
Taking $X$ and $T$ larger and larger, and using Cantor's diagonal argument, we can find a subsequence of $\{n\}$
(denoted it again by $\{n_i\}$) and a function in $C^{2+\alpha, 1+\alpha/2}_{loc} ([0,\infty) \times \R)$
(denoted it again by $\U(x,t)$) such that
\begin{equation}\label{def calU}
U_{n_i} (x,t) \to \U (x,t) \mbox{ as } i\to \infty, \quad \mbox{ in } C^{2,1}_{loc} ([0,\infty)\times \R) \mbox{ topology}.
\end{equation}
Thus we obtain an entire solution $\U$ of \eqref{p}.

%%%%%%%%%%%%%%%%%%%%%%%
%%%%%%%%%%%%%%%%%%%%%%%

\subsection{Properties of $\U$}
In this part we study the properties of $\U$ and show that it is the desired solution in Theorem \ref{thm:II}.

\subsubsection{Monotonicity and concavity}

\begin{prop}\label{prop:entire sol II-1}
Let $\U$ be the entire solution obtained in \eqref{def calU}. Then
\begin{equation}\label{Ux00 = 1/2}
\U_x (0,0)=\frac12 V'(0),
\end{equation}
\begin{equation}\label{prop of calU}
\U_t(x,t)>0 \mbox{ for } x>0,\ t\in\R,
\end{equation}
\begin{equation}\label{prop of calU2}
\U_x(x,t)>0,\ \ \U_{xx}(x,t)<0 \mbox{ for } x\geq 0,\ t\in \R,
\end{equation}
and
\begin{equation}\label{calU to V}
\U( \cdot, t) \to V(\cdot) \mbox{ as }t\to \infty,\quad \mbox{ in } L^{\infty}([0,\infty)) \mbox{ topology}.
\end{equation}
\end{prop}

\begin{proof}
The conclusions in \eqref{Ux00 = 1/2}, \eqref{prop of calU} and \eqref{prop of calU2} follow from \eqref{prop Un} and the
strong maximum principle easily.

Now we prove \eqref{calU to V}. For any small $\varepsilon>0$, since $V(\infty)=1$, there exists $X>0$ such that
$$
1-\varepsilon \leq V(x)\leq 1 \mbox{ for } x\geq X.
$$
Since $\U_t(x,t)>0$, as proving \eqref{un to V} one can show that $\U(\cdot,t)\to V(\cdot)$ as
$t\to \infty$, in the topology of $C^2_{loc} ([0,\infty))$.
Hence for some $T>0$ we have
\begin{equation}\label{0X}
\|\U(\cdot ,t)-V(\cdot)\|_{L^\infty ([0,X])} \leq \varepsilon,\quad \mbox{when }t\geq T.
\end{equation}
This implies that
$$
|\U(X,t)-1| \leq |\U(X,t)-V(X)| + |V(X)-1|\leq 2\varepsilon,\quad \mbox{when }t\geq T.
$$
Therefore, when $t\geq T$ and $x\geq X$ we have
$$
|\U(x,t)-V(x)| \leq |\U(x,t)-1| +|V(x)-1| \leq |\U(X,t)-1| +|V(X)-1| \leq 3\varepsilon.
$$
Combining with \eqref{0X}, we obtain \eqref{calU to V}.
\end{proof}

\begin{remark}\rm
Note that the first type of entire solutions $\mathcal{U}^c$ have all the properties of $\U$ except for
\eqref{Ux00 = 1/2} and $\U_{xx}(x,t)<0$. The concavity is the main difference between them.
\end{remark}

Furthermore, we can show the following properties for $\U$.

\begin{prop}\label{prop:u uxx to 0}
Let $\U(x,t)$ be the entire solution obtained as above. Then
$\U(x,t+s)\to 0$ as $s\to -\infty$, in  $C^{2,1}_{loc} ([0,\infty)\times \R)$ topology.
\end{prop}

\begin{proof}
For any time sequence $\{s_k\}$ decreasing to $-\infty$, by the parabolic estimate
as in \eqref{Schauder est} we have, for any $X,T>0$,
\begin{equation}\label{schauder est1}
\|\U (\cdot, \cdot +s_k)\|_{C^{2+\alpha, 1+\alpha/2} ([0,X]\times [-T,T])} \leq C_2,
\end{equation}
where $C_2$ depends on $X$ and $T$, but not on $k$. Hence there exist
a subsequence $\{s_{k_j}\}$ of $\{s_k\}$ and an entire solution $W(x,t)$ of \eqref{p} such that
$\U(\cdot, \cdot+s_{k_j})\to W(\cdot, \cdot)$ as $j\to \infty$, in $C^{2,1}_{loc} ([0,\infty)\times \R)$
topology.

We now show that $W_t \equiv 0$, and so $W(x,t)\equiv W(x)$ is a stationary solution of \eqref{p}.
For, otherwise, $W_t (x^*, t^*)\not= 0$ for some $(x^*,t^*)\in (0,\infty)\times \R$.
Since $\U_t (x^*, t^* +s_{k_j})>0$ for all $j$ we may assume that $W_t (x^*, t^*)= 3\delta_1 $ for
some $\delta_1 >0$. Then, for sufficiently large $j$ we have $\U_t (x^*, t^* +s_{k_j})>2\delta_1 $.
By the uniform estimate in \eqref{schauder est1} we see that, when $\varepsilon_1 \in (0,\frac12 )$ is small,
$$
\U_t (x^*, t^* + t) >\delta_1  \mbox{ if } |t-s_{k_j}|<\varepsilon_1.
$$
Without loss of generality, we assume that $s_{k_{j}} -s_{k_{j+1}} >1 $, then
\begin{eqnarray*}
\U(x^*, t^* +s_{k_j} +\varepsilon_1) & > & \U(x^*, t^* +s_{k_j} -\varepsilon_1) +2\delta_1 \varepsilon_1 \\
\ & \geq & \U(x^*, t^* + s_{k_{j+1}} +\varepsilon_1) +2\delta_1  \varepsilon_1 \geq \cdots \\
\ & \geq & \U(x^*, t^* + s_{k_{j'}} -\varepsilon_1) +2 (j'-j)\delta_1  \varepsilon_1 \to \infty, \mbox{\ \ as } j'\to \infty,
\end{eqnarray*}
contradicting the fact $\U(x,t)\in (0,1)$.  Thus $W_t\equiv 0$ and so $W(x,t)\equiv W(x)$ solves
$$
W_{xx} +f(W)=0\ (x>0),\qquad W(0)=0.
$$
This problem has only two nonnegative solutions $0$ and $V(x)$. Clearly $W(x)\not= V(x)$ since $\U_t>0$ and
$\U(x,t)\to V(x)$ as $t\to \infty$. Hence $W(x)\equiv 0$, and $\U(\cdot, \cdot+s_{k_j})\to 0$ as $j\to \infty$.

Since $\{s_k\}$ is an arbitrary sequence tending to $-\infty$, we conclude that
$\U(\cdot, \cdot+s)\to 0$ as $s\to -\infty$, in $C^{2,1}_{loc} ([0,\infty)\times \R)$
topology.
\end{proof}

\subsubsection{Uniform upper bound of $\U$: $\U(\infty, t)\equiv 1$}
Define
$$
\kappa (t):= \U(\infty, t),\quad \beta(t):= \U_x(0,t),\quad t\in \R.
$$
Since the entire solution $\U$ is obtained by taking limit for a subsequence of $\{u_n (x,t)\}$ and
each $u_n(x,t)$ takes supermum  in $(0,1)$ for any $t>0$, one may guess that $\kappa(t)$
also takes values in $(0,1)$. We will see that this in not true. In fact, in what follows we can
prove a surprising result: $\kappa(t)= 1$ for all $t\in \R$.

\begin{lem}\label{lem:k(t) 01}
Assume $ \kappa (t_1)\in (0,1)$ for some $t_1 \in \R$. Then
\begin{itemize}
\item[(i)] $\kappa'(t)>0$ for $t\in \R$, and $\kappa (t)\to 1$ as $t\to \infty$, $\kappa (t)\to 0$ as $t\to -\infty$;

\item[(ii)]for any $\nu \in (0,1)$, $ \beta(t)/ [\kappa(t)]^{1+\nu} \to \infty$ as $t\to -\infty$.
\end{itemize}
\end{lem}

\begin{proof}
(i). For any $\alpha\in (0,1)$ and $T>0$, by parabolic estimate we have $\|\U(x,\cdot)\|_{C^{1+\alpha/2} ([-T,T])} \leq C$
for some $C$ independent of $x$. Hence, $\U(x,t)\to \kappa (t)$ as $x\to \infty$ in $C^{1}_{loc} (\R)$ topology.
In the equation of $\U$, if we take limit as $x\to \infty$ in $C^1_{loc}(\R)$ topology, then we have
$$
\kappa'(t) = f(\kappa (t)) .
$$
Using $\kappa(t_1) \in (0,1)$ as the initial data, we derive the conclusions.

(ii). For any $\nu \in (0,1)$, set
$$
z(x,t;\nu) := \frac{\U(x,t)}{[\kappa(t)]^{1+\nu}},\quad x\geq 0,\ t\in \R.
$$
Then $z$ solves
$$
\left\{
 \begin{array}{ll}
 z_t = z_{xx} + c(x,t) z,& x>0, \ t\in \R,\\
 z(0,t;\nu )=0, & t\in \R,
\end{array}
\right.
$$
where
$$
c(x,t):= \frac{f(\U)}{\U} - (1+\nu) \frac{f(\kappa)}{\kappa} \leq -\nu f'(0) + O(1)\kappa(t) <0, \quad \mbox{ when } t\ll -1.
$$
Since $z_{xx} = \U_{xx} /[\kappa(t)]^{1+\nu} <0$, we conclude that
$$
z_t(x,t;\nu)<0,\quad x>0,\ t\ll -1.
$$
Therefore,
\begin{eqnarray*}
\frac{d}{dt} [z_x(0,t;\nu)] & = & z_{xt} (0,t;\nu) = \lim\limits_{x\to 0+} \frac{z_t(x,t;\nu) -z_t(0,t;\nu)}{x } \\
& = & \lim\limits_{x\to 0+}
\frac{z_t(x,t;\nu)}{x} \leq 0,\quad t\ll -1.
\end{eqnarray*}
This means that $z_x(0,t;\nu)$ in a non-increasing function when $t\ll -1$. Note that $z_x(0,t;\nu)>0$ for all $t\in \R$,
we have
$$
z_x(0,t;\nu) \geq \vartheta(\nu),\quad t\leq 0,
$$
for some positive real number $\vartheta(\nu)$. For any $\nu_1 >\nu$ we have
$$
\frac{\beta(t)}{[\kappa(t)]^{1+\nu_1}} =
\frac{\U_x (0,t)}{[\kappa(t)]^{1+\nu}} \cdot \frac{1}{[\kappa(t)]^{\nu_1 -\nu}} =
 \frac{z_x(0,t;\nu) }{[\kappa(t)]^{\nu_1 -\nu}} \geq  \frac{\vartheta(\nu)}{[\kappa(t)]^{\nu_1 -\nu}} \to \infty,\quad \mbox{as } t\to -\infty.
$$
Since $\nu\in (0,1)$ and $\nu_1 >\nu$ are arbitrary, we obtain the conclusion in (ii).
\end{proof}

Using this lemma we can prove the uniform upper bound of $\U$:

\begin{prop}\label{prop:entire sol II-2}
Assume \eqref{Fisher-KPP} and \eqref{additional cond}. Let $\U$ be the entire solution obtained in \eqref{def calU}. Then
\begin{equation}\label{calU infinity = 1}
\kappa (t) := \U(\infty, t) =1\mbox{ for all } t\in \R.
\end{equation}
\end{prop}

\begin{proof}
Set
$$
M:= \max\limits_{u\in [0,1]} |f''(u)|,\quad M_1 := \max\left\{ \frac{e^t}{\sqrt{t}}\; \left| \; t\in
J:= \Big[\frac{1}{\mu }, \frac{2}{\mu} \Big] \right. \right\}.
$$
If the conclusion is not true, then $ \kappa (t_1)\in (0,1)$ for some $t_1 \in \R$. By the above lemma,
there exists $\tau < 0$ with $|\tau|$ sufficiently large and $ \epsilon: = \kappa (\tau ) =\U(\infty, \tau) >0$ sufficiently
small such that
\begin{equation}\label{choice epsilon}
 M \epsilon < \mu^2 = f'(0),\quad 2< e^{\frac{2\mu }{M \epsilon}}, \qquad \frac{\beta(\tau)}{[\kappa(\tau)]^{3/2}} > 2\sqrt{M} M_1.
\end{equation}
We continue to define a series of parameters. Set
\begin{equation}\label{parameters}
\left\{
 \begin{array}{l}
\displaystyle b:= f'(0) -\frac{M\epsilon}{2} \in \Big( \frac12 f'(0), f'(0)\Big);\\
\displaystyle  \epsilon_1 := 2\epsilon e^{-\frac{2\mu}{M\epsilon}}\in (0,\epsilon);\\
\displaystyle  \tau_1 := \kappa^{-1} (\epsilon_1)\ \ \ (\ \Leftrightarrow\ \epsilon_1 = \kappa(\tau_1)\ );\\
\displaystyle  T_1 := \frac1b \ln\frac{2\epsilon}{\epsilon_1} = \frac{2\mu}{bM\epsilon}.
 \end{array}
\right.
\end{equation}

Choosing $X_1 >0$ large so that
$$
\U(x,\tau_1) \geq \frac12 \U(\infty, \tau_1) = \frac{\epsilon_1}{2},\quad x\geq X_1,
$$
we compare $\U$ with the solutions of the following problems:
\begin{equation}\label{IBVP-v}
\left\{
 \begin{array}{ll}
 v_t = v_{xx} + b v, & x>X_1 ,\ t>0,\\
 v(X_1 ,t)=0, & t>0,\\
 v(x,0) = \frac{\epsilon_1}{2}, & x\geq X_1,
\end{array}
\right.
\end{equation}
and
\begin{equation}\label{IBVP-w}
\left\{
 \begin{array}{ll}
 w_t = w_{xx} + \mu^2  w, & x>0,\ t>0,\\
 w(0,t)=0, & t>0,\\
 w(x,0) = \epsilon_1, & x\geq 0.
\end{array}
\right.
\end{equation}
Both problems are linear ones and can be solved explicitly. In particular,
\begin{equation}\label{express wn}
w (x,t)  =  \frac{e^{\mu^2  t}}{2\sqrt{\pi t}} \int_0^\infty \left[ e^{-\frac{(x-y)^2}{4t}} - e^{-\frac{(x+y)^2}{4t}} \right]
\epsilon_1  dy,\quad x\geq 0,\ t>0,
\end{equation}
and so
\begin{equation}\label{est IBVP wx0}
w_x (0,t)  =  \frac{e^{\mu^2 t}}{2 t \sqrt{\pi t} } \int_0^\infty  y e^{-\frac{y^2}{4t}} \epsilon_1  dy
=  \frac{\epsilon_1 e^{\mu^2 t} }{\sqrt{\pi t}},\quad t>0.
\end{equation}
Clearly, $w$ is a supersolution of \eqref{p} since $f(w)\leq \mu^2 w$ for $w\geq 0$, and so
\begin{equation}\label{Ux<wx}
\U_x(0,t+ \tau_1 ) \leq w_x(0,t) \leq \frac{\epsilon_1 e^{\mu^2 t}}{ \sqrt{\pi t}}\mbox{ for all } t>0.
\end{equation}
On the other hand, as in \eqref{express wn}, $v$ can be expressed as
\begin{eqnarray*}
v(x,t) & = & \frac{e^{bt}}{2\sqrt{\pi t}} \int_0^\infty \Big[ e^{-\frac{(x-X_1-y)^2}{4t}} -
e^{-\frac{(x-X_1+y)^2 }{4t}} \Big] \frac{\epsilon_1}{2} dy \\
& = & \frac{\epsilon_1 e^{b t}}{2 \sqrt{\pi }}
\int_{\frac{X_1 -x}{2\sqrt{t}}}^{\frac{x-X_1}{2\sqrt{t}}} e^{-s^2} ds,\quad x\geq X_1,\ t>0,
\end{eqnarray*}
and so
$$
v(x,t)\leq v(\infty, t) = \frac{\epsilon_1 e^{b t}}{2} \leq  \frac{\epsilon_1}{2}e^{bT_1}= \epsilon,\quad x\geq X_1,\ 0<t\leq T_1 .
$$
Hence, when $x\geq X_1$ and $t\in (0,T_1]$ we have
$$
f(v)\geq f'(0) v-\frac{M}{2} v^2 \geq \Big[ f'(0) -\frac{M\epsilon}{2} \Big] v =bv .
$$
This implies that $v$ is a subsolution, and so
$$
\U(x,t+ \tau_1 ) \geq v(x,t),\quad x\geq X_1,\ 0<t\leq T_1.
$$
In particular, at $t=T_1$ we have
$$
\U(\infty, T_1 + \tau_1 ) \geq v(\infty , T_1) = \frac{\epsilon_1 e^{b T_1}}{2} = \epsilon = \U(\infty, \tau).
$$
Thanks to $\U_t (\infty, t) =\kappa'(t)>0$ we have $T_1 + \tau_1 \geq \tau$.
Combining this inequality with \eqref{parameters} and \eqref{Ux<wx} we have
\begin{eqnarray*}
\beta (\tau) & = & \U_x (0,\tau) \leq \U_x (0,T_1 + \tau_1 ) \\
 & \leq &  w_x (0,T_1  )  \leq   \frac{\epsilon_1 e^{\mu^2 T_1}}{ \sqrt{\pi T_1}} \\
& =  &%\UTF{00A1}\UTF{00A1}%
\frac{\epsilon_1 e^{bT_1} e^{\frac{M\epsilon}{2} T_1}} {\sqrt{\pi T_1}}
= \frac{2 \epsilon e^{\frac{M\epsilon}{2} T_1}} {\sqrt{\pi T_1}}  \\
& \leq & 2\sqrt{M} \epsilon^{3/2} \cdot  \frac{e^{\frac{M\epsilon}{2} T_1}} {\sqrt{\frac{M\epsilon }{2}T_1}}.
\end{eqnarray*}
Since
$$
\frac{M\epsilon} {2}T_1 =\frac{\mu}{b} \in J:= \Big[ \frac{1}{\mu}, \frac{2}{\mu} \Big],
$$
we have
$$
\frac{e^{\frac{M\epsilon}{2} T_1}} {\sqrt{\frac{M\epsilon }{2}T_1}} \leq M_1
$$
by the definition of $M_1$, and so
$$
\beta(\tau) \leq 2\sqrt{M} M_1 \epsilon^{3/2} = 2\sqrt{M} M_1 [\kappa(\tau)]^{3/2},
$$
contradicting \eqref{choice epsilon}. This proves Proposition \ref{prop:entire sol II-2}.
\end{proof}

\subsection{The limit of $\U$ as $t\to -\infty$ and the proof of Theorem \ref{thm:II}}
By the properties of $\U$ obtained above we see that, for any $m\in (0,1)$, the $m$-level set of
$\U(\cdot, s)$ is a unique point $\xi_m (s)$, that is, $x=\xi_m (s)$ is the unique root of $\U(x,s)=m$
for each $s\in \R$. Note that, without the above proposition, that is, if $\kappa(s)\to 0\ (s\to -\infty)$,
then $\U(\cdot,s)$ may have no $m$-level set when $s\ll -1$.

\begin{lem}\label{lem:ximt to infty}
$\xi_m (s)\to \infty$ as $s\to -\infty$.
\end{lem}

\begin{proof}
By the definition we have
$$
m=\U(\xi_m (s), s) -\U(0,s) =\U_x(\xi^* , s)\cdot \xi_m (s) \leq \U_x(0,s) \cdot \xi_m (s),\quad \mbox{for some }
\xi^* \in (0,\xi_m (s)).
$$
Since $\U_x(0,s)\to 0\ (s\to -\infty)$ by Proposition \ref{prop:u uxx to 0}, we have
$\xi_m(s) \to \infty$ as $s\to -\infty$.
\end{proof}

\begin{prop}\label{prop:calU to eta-m}
Let $\eta_m(t)$ be the solution of \eqref{ivp}. Then $\U(x+\xi_m(s), t+s)\to \eta_m(t)$ as $s\to -\infty$,
in $C^{2,1}_{loc}(\R\times \R)$ topology. In addition, $\xi'_m (s) \to -\infty$ as $s\to -\infty$.
\end{prop}

\begin{proof}
For any time sequence $\{s_k\}$ decreasing to $-\infty$ and any $M>0,\ T>0$, by the parabolic estimate we have
$$
\|\U(x+\xi_m(s_k ), t+ s_k )\|_{C^{2+\alpha, 1+\alpha/2} ([-M,M]\times [-T,T])} \leq C,
$$
for some constant $C$ depending on $M$ and $T$ but not on $k$. Hence there is a subsequence $\{s_{k_j}\}$
of $\{s_k\}$  and a function $W_m (x,t)\in C^{2+\alpha,1+\alpha/2}_{loc} (\R\times \R)$ such that
$$
\U(x+\xi_m (s_{k_j}), t+ s_{k_j} )\to W_m (x,t) \mbox{ as } j\to \infty, \quad \mbox{ in } C^{2,1}_{loc}(\R\times \R) \mbox{ topology}.
$$
Clearly,
\begin{equation}\label{wm 00}
W_m(0,0) = \lim\limits_{j\to \infty} \U(\xi_m (s_{k_j}), s_{k_j} ) =m.
\end{equation}

Note that, for any $x\in [-M,M]$ and $s\in [-T,T]$,
\begin{eqnarray*}
\ & & \U(x+\xi_m (s_{k_j}), t+ s_{k_j} ) \\
& = & \U(x+\xi_m (s_{k_j}), t+ s_{k_j} ) - \U(\xi_m (s_{k_j}), t+ s_{k_j} ) + \U(\xi_m (s_{k_j}), t+ s_{k_j} )\\
& = & \U_x (\rho  x+\xi_m (s_{k_j}), t+ s_{k_j} ) \cdot x +  \U(\xi_m (s_{k_j}), t+ s_{k_j}),\quad \mbox{ for some } \rho\in (0,1).
\end{eqnarray*}
Since
$$
| \U_x (\rho x+\xi_m (s_{k_j}), t+ s_{k_j} ) \cdot x| \leq \U_x (0,t+  s_{k_j} ) \cdot M \to 0  \mbox{ as }j\to \infty,
$$
and
$$
\U(\xi_m (s_{k_j}), t+ s_{k_j} ) \to W_m(0,t) \mbox{ as }j\to \infty,
$$
we have
$$
\U(x+\xi_m (s_{k_j}), t+ s_{k_j} ) \to W_m(0,t) \mbox{ as }j\to \infty.
$$
Hence
$$
W_m(x,t) \equiv W_m(0,t),\quad x, t\in \R.
$$
In other words, $W_m(x,t)$ is actually independent of $x$. Consequently,
$$
\U_{xx}(x+\xi_m (s_{k_j}), t+ s_{k_j} )\to (W_m)_{xx} (x,t) \equiv 0\mbox{ as } j\to \infty, \quad \mbox{ in } L^\infty_{loc}
(\R\times \R) \mbox{ topology}.
$$

Now we show that $W_m(x,t)\equiv W_m(0,t)$ is nothing but the solution $\eta_m(t)$  of \eqref{ivp}.
In the equation
$$
\U_t (\xi_m (s_{k_j}), t+ s_{k_j} ) = \U_{xx} (\xi_m (s_{k_j}), t+ s_{k_j} ) + f(\U(\xi_m (s_{k_j}), t+ s_{k_j} )),
$$
by taking limit as $j\to \infty$  we have
$$
[W_m (0,t)]_t  = f(W_m(0,t)).
$$
Combining with \eqref{wm 00} we see that $W_m(x,t)\equiv W_m(0,t)$ is the unique solution $\eta_m(t)$ of \eqref{ivp}.

Since the limit function $\eta_m(t)$ is unique we conclude that
\begin{equation}\label{calU to eta t to -infty}
\U(x+\xi_m (s), t + s) -\eta_m(t) \to 0\mbox{ as } s\to -\infty,\quad \mbox{ in } C^{2,1}_{loc}(\R\times \R) \mbox{ topology}.
\end{equation}
In particular, at the point $(x,t)=(0,0)$, we have
\begin{equation}\label{limits of Ux Ut}
\U_x (\xi_m (s), s)\to 0,\quad \U_t(\xi_m (s), s)\to \eta'_m (0)= f(\eta_m (0)) =f(m)>0\quad \mbox{ as } s\to -\infty.
\end{equation}
By differentiating $\U(\xi_m(s),s)\equiv m$ with respect to $s$ we have
$$
\U_x(\xi_m(s),s) \cdot \xi'_m(s) + \U_t(\xi_m(s),s) \equiv 0.
$$
This reduces to $\xi'_m(s) \to -\infty\ (s\to -\infty)$ by \eqref{limits of Ux Ut},
and then the proposition is proved.
\end{proof}

\medskip
\noindent
{\it Proof of Theorem \ref{thm:II}}. The conclusions follow from Propositions \ref{prop:entire sol II-1} and \ref{prop:calU to eta-m}
directly. \hfill $\Box$

\begin{remark}\rm
In the locally uniform topology, the entire solution $\U(x,t)$ obtained in this section can be
regarded as a heteroclinic orbit connecting the identically zero and $V(x)$,
 while the spatially uniform solution $\eta_m(t)$ is a heteroclinic
orbit connecting the constant solution $0$ and $1$ to the equation in the whole space.
At $x=\infty$, however, $\U(x,t)$ always takes $1$. It seems that this is caused by the effect of the Dirichlet boundary condition.
On the other hand, the odd symmetric extension of the solutions around the origin give rise to
the symmetric solutions to the Allen-Cahn equation (that is, the reaction-diffusion equation with balanced bistable nonlinearity)
in the whole space.
Hence, our result also contributes to the study of
entire solutions of the Allen-Cahn equation.
\end{remark}

\medskip
\noindent
{\bf Acknowledgement}. The authors would like to thank the anonymous referees for their
valuable suggestions.

%%%%%%%%%%%%%%%%%%%%%%%%%%%%%%%%%%%%%%
%%%%%%%%%%%%%%%%%%%%%%%%%%%%%%%%%%%%%%

\end{document}